\documentclass[reqno,12pt]{amsart}
\theoremstyle{plain}
\newtheorem{thm}{Theorem}[section]
\newtheorem{lemma}[thm]{Lemma} 
\newtheorem{prop}[thm]{Proposition}

\theoremstyle{remark}

\theoremstyle{definition}
\newtheorem{defi}[thm]{Definition}

\usepackage{amscd,amssymb,comment,euscript}
\usepackage{pict2e}
\usepackage{enumitem}
\usepackage{multirow}
\usepackage[initials]{amsrefs}

\newcommand\CO{{\operatorname{CO}}}
\newcommand\Cpx{{\mathbb C}}
\newcommand\NC{\operatorname{NC}}
\newcommand\Pc{{\mathcal{P}}}
\newcommand\PC{{\operatorname{PC}}}
\newcommand\Phit{{\widetilde\Phi}}
\newcommand\phit{{\tilde\phi}}
\newcommand\pihat{{\hat\pi}}
\newcommand\pit{{\tilde\pi}}
\newcommand\Psit{{\widetilde\Psi}}
\newcommand\sigmahat{{\hat\sigma}}
\newcommand\sigmat{{\tilde\sigma}}
\newcommand\simpi{\overset{\pi}{\sim}}
\newcommand\simsigma{\overset{\sigma}{\sim}}
\newcommand\Thetat{{\widetilde\Theta}}
\newcommand\tr{{\mathrm{tr}}}

\begin{document}

\title[purely crossing]{Generating functions for purely crossing partitions}

\author[Dykema]{Kenneth J.\ Dykema}
\address{Department of Mathematics, Texas A\&M University,
College Station, TX 77843-3368, USA}
\email{kdykema@math.tamu.edu}
\thanks{Research supported in part by NSF grant DMS-1202660.}


\date{February 14, 2016}

\begin{abstract}
The generating function for the number of  purely crossing partitions of $\{1,\ldots,n\}$
is found in terms of the generating function for Bell numbers.
Further results about generating functions for asymptotic moments of certain random Vandermonde matrices are derived.
\end{abstract}

\maketitle

\section{Introduction}

Let $\Pc(n)$ denote the lattice of all set partitions of $[n]:=\{1,\ldots,n\}$.
The elements of $\pi\in\Pc(n)$ are called blocks of the partition $\pi$;
we write $i\simpi j$ if and only if $i$ and $j$ belong to the same block of $\pi$ and we write $i\not\simpi j$
if and only if they belong to different blocks of $\pi$.

Recall that a partition $\pi\in\Pc(n)$ is said to be noncrossing whenever there do not exist $i,j,k,l\in[n]$
with $i<j<k<l$, $i\simpi k$ and $j\simpi l$, but $i\not\simpi j$.
We let $\NC(n)$ denote the set of all noncrossing partitions of $[n]$.

\begin{defi}\label{defi:split}
We say that a subset $S\subseteq[n]$ {\em splits} a partition $\pi\in\Pc(n)$ if $S$ is the union of some of the blocks of $\pi$.
In other words, $S$ splits $\pi$ if and only if $B\in\pi$ and $B\cap S\ne\emptyset$ implies $B\subseteq S$.
\end{defi}

\begin{defi}\label{defi:PC}
For $n\ge1$, a partition $\pi\in\Pc(n)$ is said to be {\em purely crossing} if
\begin{enumerate}[label=(\alph*),leftmargin=20pt]
\item\label{it:nosplit} no proper subinterval $\{p+1,p+2,\ldots,p+q\}$ of $[n]$ splits $\pi$
(by proper subinterval we mean with $0\le p<p+q\le n$ and $q<n$)
\item\label{it:noneigh} no block of $\pi$ contains neighbors, namely, $k\not\simpi k+1$ for all $k\in[n-1]$
\item\label{it:no1n} $1\not\simpi n$.
\end{enumerate}
We let $\PC(n)$ denote the set of all purely crossing partitions of $[n]$
and we let $\PC=\bigcup_{n=1}^\infty\PC(n)$.
\end{defi}
Note that condition~\ref{it:nosplit} implies that $\pi$ has no singleton blocks.
Moreover, if $\pi\in\PC(n)$, then every partition obtained from $\pi$ by cyclic permutation of $[n]$
is also purely crossing.

The purely crossing partitions were introduced in~\cite{BD} in connection with certain random Vandermonde matrices.
Is it easy to see that $\PC(n)$ is empty for $n\in\{1,2,3,5\}$ and that 
the only purely crossing partitions of sizes $n\in\{4,6,7\}$ are those found in Table~\ref{tab:467} and 
their orbits under cyclic permutations of $[n]$.
\begin{table}[ht]
\caption{Purely crossing partitions of sizes $4$, $6$ and $7$}
\label{tab:467}
\begin{center}
\begin{tabular}{c||c|l|c} 
$n$ & picture & blocks & \# in orbit \\ \hline\hline
$4$ & \begin{picture}(30,13)(0,0)
 \multiput(0,0)(20,0){2}{\line(0,1){5}}
 \multiput(10,0)(20,0){2}{\line(0,1){9}}
 \put(0,5){\line(1,0){20}}
 \put(10,9){\line(1,0){20}}
 \end{picture}
   & $\{1,3\},\{2,4\}$ & 1 \\ \hline\hline
\multirow{3}{*}{$6$}
 & \begin{picture}(50,13)(0,0)
 \multiput(0,0)(20,0){3}{\line(0,1){5}}
 \multiput(10,0)(20,0){3}{\line(0,1){9}}
 \put(0,5){\line(1,0){40}}
 \put(10,9){\line(1,0){40}}
 \end{picture}
   & $\{1,3,5\},\{2,4,6\}$ & 1 \\ \cline{2-4}
 & \begin{picture}(50,13)(0,0)
 \multiput(0,0)(20,0){2}{\line(0,1){5}}
 \multiput(30,0)(20,0){2}{\line(0,1){5}}
 \multiput(10,0)(30,0){2}{\line(0,1){9}}
 \multiput(0,5)(30,0){2}{\line(1,0){20}}
 \put(10,9){\line(1,0){30}}
 \end{picture}
   & $\{1,3\},\{2,5\},\{4,6\}$ & 3 \\ \cline{2-4}
 & \begin{picture}(50,14)(0,0)
 \multiput(0,0)(30,0){2}{\line(0,1){4}}
 \multiput(10,0)(30,0){2}{\line(0,1){7}}
 \multiput(20,0)(30,0){2}{\line(0,1){10}}
 \put(0,4){\line(1,0){30}}
 \put(10,7){\line(1,0){30}}
 \put(20,10){\line(1,0){30}}
 \end{picture}
   & $\{1,4\},\{2,5\},\{3,6\}$ & 1 \\ \hline\hline
\multirow{2}{*}{$7$}
 & \begin{picture}(60,13)(0,0)
 \multiput(0,0)(20,0){2}{\line(0,1){5}}
 \put(50,0){\line(0,1){5}}
 \multiput(10,0)(20,0){2}{\line(0,1){9}}
 \multiput(40,0)(20,0){2}{\line(0,1){9}}
 \put(0,5){\line(1,0){50}}
 \multiput(10,9)(30,0){2}{\line(1,0){20}}
 \end{picture}
   & $\{1,3,6\},\{2,4\},\{5,7\}$ & 7 \\ \cline{2-4}
 & \begin{picture}(60,14)(0,0)
 \multiput(0,0)(20,0){2}{\line(0,1){4}}
 \put(50,0){\line(0,1){4}}
 \multiput(10,0)(30,0){2}{\line(0,1){7}}
 \multiput(30,0)(30,0){2}{\line(0,1){10}}
 \put(0,4){\line(1,0){50}}
 \put(10,7){\line(1,0){30}}
 \put(30,10){\line(1,0){30}}
 \end{picture}
   & $\{1,3,6\},\{2,5\},\{4,7\}$ & 7 \\ \hline\hline
\end{tabular}
\end{center}
\end{table}

In this note, we study $\PC(n)$
and describe how  arbitrary partitions can be realized in terms of purely crossing ones and noncrossing partitions.
Thereby, we find an algebraic description of the generating function for $|\PC(n)|$ in terms of the generating function for
Bell numbers.

More generally, in Section~\ref{sec:partitions}
we study generating functions based on certain weighted sums over partitions.
An intermediate stage is to consider the connected partitions, which have been enumerated by F. Lenher~\cite{L02},
who proved that their cardinalities form the free cumulants of a Poisson distribution and found a generating series for them.
We will make an equivalent derivation of the generating series, thereby reproving Lehner's result, but with an eye toward
generalization that follows.
A motivation for and a possible application of this study in random matrices
are described in Section~\ref{sec:reasons}.
This includes more complicated results about algebra-valued generating functions that are similar
to those in Section~\ref{sec:partitions}.

\noindent
{\bf Notation.}
We will refer to the elements of $[n]$ as the atoms of the partition $\pi\in\Pc(n)$.
We let $1_n=\{[n]\}$  be the partition of $[n]$ into one block and
$0_n=\{\{1\},\{2\},\ldots,\{n\}\}$ be the partition of $[n]$ into $n$ blocks (all singletons).

If $\phi:A\to B$ is a bijective mapping and if $\pi$ is a set partition of $A$, then (by abuse of notation)
we let $\phi(\pi)$ denote the set partition of $B$ that results from applying $\phi$ to every block of $\pi$.

The usual order $\le$ on $\Pc(n)$ is $\sigma\le\pi$ if and only if every block of $\pi$ splits $\sigma$.

If $\pi$ is a partition of a set $X$, then the restriction of $\pi$ to a subset $Y\subseteq X$ is
$\{Y\cap B\mid B\in X\}\backslash\{\emptyset\}$.

\noindent
{\bf Acknowledgement.}
The author thanks Catherine Yan for helpful discussions about generating functions.

\section{Generating functions and NIS partitions}
\label{sec:partitions}

For each $n\ge1$ and $\pi\in\PC(n)$, let $a(\pi)\in\Cpx$.
Consider the formal power series
\[
A(x)=\sum_{n=1}^\infty a_nx^n=a_4x^4+a_6x^6+a_7x^7+\cdots,\quad\text{where}\quad a_n=\sum_{\pi\in\PC(n)}a(\pi).
\]
Note that, if $a(\pi)=1$ for all $\pi\in\PC(n)$, then
\[
A(x)=\sum_{n=1}^\infty |\PC(n)|x^n
\]
is the generating series for $|\PC(n)|$.
This will be the principal case of interest in this paper;
however, we prove results in the greater generality of
arbitrary coefficients $a(\pi)$.

\begin{defi}
Let $\PC^+(n)$ be the set of all $\pi\in\Pc(n)$ such that~\ref{it:nosplit} and~\ref{it:noneigh} of Definition~\ref{defi:PC} hold.
\end{defi}

The following lemma is straightforward and we omit a proof.
\begin{lemma}\label{lem:PC+}
$\PC^+(1)=\{0_1\}$ and, for $n\ge2$, we have the disjoint union
\[
\PC^+(n)=\PC(n)\cup\{\sigmat\mid\sigma\in\PC(n-1)\},
\]
where $\sigmat$ is obtained from $\sigma$ by adjoining $n$ to the block of $\sigma$ that contains $1$.
\end{lemma}

Let $b(0_1)=1$ and 
for $\pi\in\PC^+(n)$, $n\ge2$, let 
\[
b(\pi)=\begin{cases}
a(\pi)&\pi\in\PC(n), \\
a(\sigma)&\pi=\sigmat,\,\sigma\in\PC(n-1).
\end{cases}
\]
Consider the formal power series
\[
B(x)=\sum_{n=1}^\infty b_nx^n=x+b_4x^4+b_5x^5+\cdots,\quad\text{where}\quad b_n=\sum_{\pi\in\PC^+(n)}b(\pi).
\]
Thus, if $a(\pi)=1$ for all $\pi\in\PC(n)$ and all $n$, then
\[
B(x)=\sum_{n=1}^\infty|\PC^+(n)|x^n.
\]

\begin{lemma}\label{lem:AB}
We have $b_1=1$ and for every $n\ge2$,
$b_n=a_n+a_{n-1}$.
Thus, we have
\[
B(x)=x+(1+x)A(x).
\]
\end{lemma}
\begin{proof}
By definition of $\PC^+(1)$, $b_1=1$.
For $n\ge2$, by Lemma~\ref{lem:PC+}, we have
\[
b_n=\sum_{\pi\in\PC(n)}a(\pi)+\sum_{\sigma\in\PC(n-1)}a(\sigma)=a_n+a_{n-1}.
\]
The desired equality follows immediately.
\end{proof}

\begin{defi}
For $n\ge1$, let $\CO(n)$ be the set of all connected partitions $\pi\in\Pc(n)$,
namely, such that~\ref{it:nosplit} of Definition~\ref{defi:PC} holds.
\end{defi}

\begin{defi}\label{def:pihat}
Given $\pi\in\Pc(n)$, let $\pihat$ denote the smallest noncrossing partition so that $\pi\le\pihat$,
in the usual order on $\Pc(n)$.
\end{defi}

\begin{lemma}\label{lem:NIS-NC}
Let $\pi\in\Pc(n)$.
Then $\pi\in\CO(n)$ if and only if $\pihat=1_n$.
\end{lemma}
\begin{proof}
If $\pi\notin\CO(n)$, then there is a proper subinterval $I$ of $[n]$ that splits $\pi$.
The partition $\{I,[n]\backslash I\}$ of $[n]$ is noncrossing and lies above $\pi$ in the usual order on $\Pc(n)$.
Thus, $\pihat\le\{I,[n]\backslash I\}$ and $\pihat\ne1_n$.

Now suppose $\pihat\ne 1_n$.
Then at least one block of $\pihat$ is a proper subinterval of $[n]$, and this block necessarily splits $\pi$.
So $\pi\notin\CO(n)$.
\end{proof}

\begin{lemma}\label{lem:NISbijection}
Let $n\ge1$.
Then there is a bijection
\[
\Phi_n:\bigcup_{1\le\ell\le n}\left\{(\sigma,k_1,\ldots,k_\ell\middle|
\sigma\in\PC^+(\ell),\,k_1,\ldots,k_\ell\ge1,\,k_1+\cdots+k_\ell=n\right\}\to\CO(n),
\]
defined by $\pi=\Phi_n(\sigma,k_1,\ldots,k_\ell)$ is the partition obtained from $\sigma$ by replacing the $j$-th atom of $\sigma$
with an interval of length $k_j$.
More precisely, for the intervals
\[
I_j=\{k_1+\cdots+k_{j-1}+1,k_1+\cdots+k_{j-1}+2,\ldots,k_1+\cdots+k_{j-1}+k_j\}
\]
we have
\[
\pi=\{\cup_{j\in B}I_j\mid B\in\sigma\}.
\]
\end{lemma}
\begin{proof}
Let $\pi=\Phi_n(\sigma,k_1,\ldots,k_\ell)$.
Suppose $K\subseteq[n]$ is a nonempty interval of $[n]$ that splits $\pi$.
Since each block of $\pi$ is a union of some of the intervals $(I_j)_{j\in[\ell]}$, we have $K=\bigcup_{j\in X}I_j$ for some
subset $X$ of $[\ell]$.
Since $K$ is a nonempty interval, $X$ must be a nonempty interval of $[\ell]$.
Since $K$ splits $\pi$, we see that $X$ splits $\sigma$.
Since $\sigma\in\PC^+(\ell)\subseteq\CO(\ell)$, we have $X=[\ell]$, so $K=[n]$.
This shows that $\Phi_n$ maps into $\CO(n)$.

We now describe the inverse map of $\Phi_n$.
Given $\pi\in\CO(n)$, consider the collection of interval subsets of $[n]$ that
are maximal with respect to the property of being contained in some block of $\pi$.
The collection of these is an interval partition of $[n]$, and we may arrange them
in increasing order $(I_j)_{1\le j\le\ell}$.
Let $\sigma$ be the partition of $[\ell]$ given by $j_1\simsigma j_2$ if and only if $I_{j_1}$ and $I_{j_2}$
belong to the same block as $\pi$.
We have $\sigma\in\PC^+(\ell)$ because if $j\simsigma j+1$, then $I_j\cup I_{j+1}$ would be the subset of a single block of $\pi$,
contradicting maximality of $I_j$.
Letting $k_j=|I_j|$, we have $\sum_{j=1}^\ell=n$.
Let $\Phit_n(\pi)=(\sigma,k_1,\ldots,k_\ell)$.
Then
\[
\Phit_n:\CO(n)\to\bigcup_{1\le\ell\le n}\left\{(\sigma,k_1,\ldots,k_\ell)\middle|
\sigma\in\PC^+(\ell),\,k_1,\ldots,k_\ell\ge1,\,k_1+\cdots+k_\ell=n\right\}
\]
and $\Phi_n\circ\Phit_n$ is the identity map on $\CO(n)$.
We easily verify that $\Phit_n\circ\Phi_n$ is the identity map on the domain of $\Phi_n$.
\end{proof}

For $\pi\in\CO(n)$, let $c(\pi)=b(\sigma)$, where $\pi=\Phi_n(\sigma,k_1,\ldots,k_\ell)$.
Let
\[
C(x)=\sum_{n=1}^\infty c_nx^n,\quad\text{where}\quad c_n=\sum_{\pi\in\CO(n)}c(\pi).
\]
Note that, if $a(\pi)=1$ for all $\pi\in\bigcup_{n=1}^\infty\PC(n)$, then
\[
C(x)=\sum_{n=1}^\infty|\CO(n)|x^n
\]
is the generating series for $|\CO(n)|$.

\begin{lemma}\label{lem:BC}
As formal power series, we have
\[
C(x)=B\left(\frac x{1-x}\right).
\]
\end{lemma}
\begin{proof}
From the bijection described in Lemma~\ref{lem:NISbijection} and the definition of $c(\pi)$, we have
\[
c_nx^n=\sum_{\pi\in\CO(n)}c(\pi)x^n=\sum_{\ell=1}^n\sum_{\sigma\in\PC^+(\ell)}b(\sigma)
\sum_{\substack{k_1,\ldots,k_\ell\ge1\\k_1+\cdots+k_\ell=n}}x^n
=\sum_{\ell=1}^nb_\ell\sum_{\substack{k_1,\ldots,k_\ell\ge1\\k_1+\cdots+k_\ell=n}}\prod_{j=1}^\ell x^{k_j},
\]
so 
\[
C(x)=\sum_{n=1}^\infty c_nx^n=
\sum_{\ell=1}^\infty b_\ell\bigg(\sum_{k=1}^\infty x^k\bigg)^\ell
=B\left(\frac x{1-x}\right).
\]
\end{proof}

\begin{lemma}\label{lem:Pbijection}
For each $n\ge1$, there is a bijection
\[
\Psi_n:\left\{(\tau,\sigma_1,\sigma_2,\ldots,\sigma_{|\tau|})\mid\tau=\{B_1,\ldots,B_{|\tau|}\}\in\NC(n),\,\sigma_j\in\CO(|B_j|)\right\}
\to\Pc(n),
\]
defined as follows.
The idea is to replace each block $B_j$ of $\tau$ with the corresponding partition of $B_j$ determined by $\sigma_j$.
For specificity, the blocks $B_1,\ldots,B_{|\tau|}$ above are written in order of increasing minimal elements.
Given $B\subseteq[n]$, let $\psi_B:[|B|]\to B$ be the unique order-preserving bijection.
Given $(\tau,\sigma_1,\ldots,\sigma_{|\tau|})$ in the domain of $\Psi_n$ and writing $\tau=\{B_1,\ldots,B_{|\tau|}\}$ as above,
we consider the partition $\psi_{B_j}(\sigma_j)$
of $B_j$ for each $j$.
Then
\[
\Psi(\tau,\sigma_1,\ldots,\sigma_{|\tau|})=\bigcup_{j=1}^{|\tau|}\psi_{B_j}(\sigma_j).
\]
\end{lemma}
\begin{proof}
Given $\pi\in\Pc(n)$, let $\pihat$ be as in Definition~\ref{def:pihat} be the noncrossing cover of $\pi$.
For each block $B$ of $\pihat$, the restriction of $\pi$ to $B$ yields a partition that, after applying
the order-preserving bijection of $B$ onto $\{1,\ldots,|B|\}$, yields a partition $\sigma_B\in\Pc(|B|)$
with $\sigmahat_B=1_{|B|}$, since otherwise $B$ would not be a block of $\pihat$.
Thus, by Lemma~\ref{lem:NIS-NC}, $\sigma_B\in\CO(|B|)$.
Consequently, if we number the blocks $\pihat$ in conventional order, $\pihat=\{B_1,B_2,\ldots,B_{|\pihat|}\}$,
then we immediately see $\pi=\Psi_n(\pihat,\sigma_{B_1},\ldots,\sigma_{B_{|\pihat|}})$.
We define
\[
\Psit_n:\Pc(n)\to\left\{(\tau,\sigma_1,\ldots,\sigma_{|\tau|})\mid\tau=\{B_1,\ldots,B_{|\tau|}\}\in\NC(n),\,\sigma_j\in\CO(|B_j|)\right\}
\]
by $\Psit_n(\pi)=(\pihat,\sigma_{B_1},\ldots,\sigma_{B_{|\pihat}})$ as given above.
Thus, the composition $\Psi_n\circ\Psit_n$ is the identity on $\Pc(n)$.
Similarly, it is not difficult to 
see that $\Psit_n\circ\Psi_n$ is the identity on the domain of $\Psi_n$.
\end{proof}

For $\pi=\Psi_n(\tau,\sigma_1,\ldots,\sigma_{|\tau|})\in\Pc(n)$, let
\begin{equation}\label{eq:d}
d(\pi)=\prod_{j=1}^{|\tau|}c(\sigma_j).
\end{equation}
Let
\[
D(x)=1+\sum_{n=1}^\infty d_nx^n,\quad\text{where}\quad d_n=\sum_{\pi\in\Pc(n)}d(\pi).
\]
For convenience and consistency, we are using the convention $\Pc(0)=\{\emptyset\}$ and we set $d(\emptyset)=1$.
Note that if $a(\pi)=1$ for all $\pi\in\PC$, then $d(\pi)=1$ for all $\pi\in\bigcup_{n\ge1}\Pc(n)$.
In this case,
$d_n=|\Pc(n)|$ is the $n$-th Bell number and $D(x)$ is the generating function
for the Bell numbers.

\begin{lemma}\label{lem:PPbijection}
For each $n\ge1$, there is a bijection
\[
\Theta_n:\Pc(n)\to
\bigcup_{\substack{1\le\ell\le n\\k(1),\ldots,k(\ell)\ge0\\k(1)+\cdots+k(\ell)=n-\ell}}
\left\{(\sigma,\pi_1,\pi_2,\ldots,\pi_\ell)\mid\sigma\in\CO(\ell),\,\pi_j\in\Pc(k(j))\right\},
\]
defined as follows.
Given $\pi\in\Pc(n)$, let $\pihat$ be as in Definition~\ref{def:pihat} and let $B_1$ be the block of $\pihat$ containing $1$.
Let $\ell=|B_1|$ and write $B_1=\{s(1),s(2),\ldots,s(\ell)\}$ with $s(j-1)<s(j)$.
Let $\phi$ be the order-preserving bijection from $B_1$ onto $[\ell]$ and let $\sigma$ be the result of $\phi$ applied to the restriction
of $\pi$ to $B_1$.
For $1\le j\le\ell-1$, let $\pi_j$ be the result of applying the mapping $i\mapsto i-s(j)$ to the restriction of $\pi$ to the interval
$\{s(j)+1,\ldots,s(j+1)-1\}$.
Let $\pi_\ell$ be the result of applying the mapping $i\mapsto i-s(\ell)$ to the interval $\{s(\ell)+1,\ldots,n\}$.
Then $\Theta_n(\pi)=(\sigma,\pi_1,\ldots,\pi_\ell)$.
\end{lemma}
\begin{proof}
Take $k(j)=s(j+1)-s(j)-1$ if $1\le j\le\ell-1$ and let $k(l)=n-s(\ell)$.
We have $\sigmahat=1_{|B_1|}$, so by Lemma~\ref{lem:NIS-NC}, $\sigma\in\CO(|B_1|)$ and
$\Theta_n$ takes values in the indicated range.

We define
\[
\Thetat_n:\bigcup_{\substack{1\le\ell\le n\\k(1),\ldots,k(\ell)\ge0\\k(1)+\cdots+k(\ell)=n-\ell}}
\left\{(\sigma,\pi_1,\pi_2,\ldots,\pi_\ell)\mid\sigma\in\CO(\ell),\,\pi_j\in\Pc(k(j))\right\}
\to\Pc(n)
\]
as follows.
Given $(\sigma,\pi_1,\ldots,\pi_\ell)$ in the indicated domain of $\Thetat_n$, let
$s(j)=k(1)+k(2)+\cdots+k(j-1)+j$, with $s(1)=1$, and let
\[
L=\{s(1),s(2),\ldots,s(\ell)\}
\]
let $\phit:[\ell]\to L$ be the order-preserving bijection and consider the partition $\phit(\sigma)$ of $L$.
For each $j$ with let
$\pit_j$ be the partition of the interval $\{s(j)+1,\ldots,s(j)+k(j)\}$ obtained by applying the mapping $i\mapsto i+s(j)$ to $\pi_j$.
Let 
\[
\Theta_n(\sigma,\pi_1,\ldots,\pi_\ell)=\phit(\sigma)\cup\bigcup_{j=1}^\ell\pit_j.
\]
It it not difficult to show that $\Theta_n$ and $\Thetat_n$ are inverses of each other.
\end{proof}

\begin{lemma}\label{lem:dTheta}
Suppose $\pi\in\Pc(n)$ and $\Theta_n(\pi)=(\sigma,\pi_1,\ldots,\pi_\ell)$.
Then, with $d$ as defined in~\eqref{eq:d}, we have
\[
d(\pi)=c(\sigma)\prod_{j=1}^\ell d(\pi_j).
\]
\end{lemma}
\begin{proof}
Let $\pi=\Psi_n(\tau,\sigma_1,\ldots,\sigma_{|\tau|})$.
Then $\sigma=\sigma_1$ and 
we draw the first block $B_1$ of $\tau$
in Figure~\ref{fig:B1ncp},
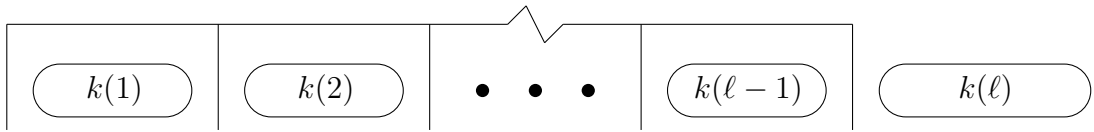
\begin{figure}[ht]
\caption{The first block of a noncrossing partition.}
\label{fig:B1ncp}
\begin{picture}(430,70)(0,0)
\put(10,50){\line(1,0){189.5}}
\put(220.5,50){\line(1,0){109.5}}
\put(199.5,50){\line(1,1){7}}
\put(206.5,57){\line(1,-2){7}}
\put(213.5,43){\line(1,1){7}}
\put(10,10){\line(0,1){40}}
\put(90,10){\line(0,1){40}}
\put(170,10){\line(0,1){40}}
\put(250,10){\line(0,1){40}}
\put(330,10){\line(0,1){40}}
\put(50,25){\oval(60,20)}
\put(40,22){$k(1)$}
\put(130,25){\oval(60,20)}
\put(120,22){$k(2)$}
\put(290,25){\oval(60,20)}
\put(270,22){$k(\ell-1)$}
\put(380,25){\oval(80,20)}
\put(370,22){$k(\ell)$}
\multiput(190,25)(20,0){3}{\circle*{5}}
\end{picture}
\end{figure}
with the gaps between elements of $B_1$ of length $k(1),\ldots,k(\ell-1)$ and with $k(\ell)$ elements of $[n]$ to the right of $B_1$.
The other blocks of $\tau$ all lie in the gaps depicted by the ovals, either between the adjacent elements of $B_1$ or to the right of $B_1$.
In particular, each $\pi_j$ is formed from those in the list $\sigma_2,\ldots,\sigma_{|\tau|}$ whose corresponding blocks of $\tau$
lie in the oval labeled $k(j)$.
Thus, we get
\[
\prod_{j=1}^\ell d(\pi_j)=\prod_{i=2}^{|\tau|}c(\sigma_j).
\]
By the definition~\eqref{eq:d} of $d(\pi)$, this proves the lemma.
\end{proof}

\begin{lemma}\label{lem:CD}
As formal power series, we have
\begin{equation}\label{eq:CD}
D(x)=1+C(xD(x)).
\end{equation}
Letting $F(x)=xD(x)$ and letting $F^{\langle-1\rangle}$ denote the inverse with respect to composition of $F$, we have
\begin{equation}\label{eq:Finv}
F^{\langle-1\rangle}(w)=\frac w{1+C(w).}
\end{equation}
\end{lemma}
\begin{proof}
Using the bijection from Lemma~\ref{lem:PPbijection} and using Lemma~\ref{lem:dTheta}, we have
\begin{align*}
D(x)&=1+\sum_{\ell=1}^\infty\left(\sum_{\sigma\in\CO(\ell)}c(\sigma)x^\ell\right)\sum_{k(1),\ldots,k(\ell)\ge0}
\;\;\sum_{\substack{\pi_1\in\Pc(k(1)),\ldots\\ \hspace{1em}\pi_\ell\in\Pc(k(\ell))}}\;\;\prod_{q=1}^\ell d(\pi_q)x^{k(q)} \\
&=1+\sum_{\ell=1}^\infty c_\ell x^\ell\bigg(\sum_{k=0}^\infty d_k x^k\bigg)^\ell
=1+C(xD(x)).
\end{align*}
Multiplying both sides of~\eqref{eq:CD} by $x$, we get
$F(x)=x(1+C(F(x)))$, and letting $w=F(x)$ we get
$w=F^{\langle-1\rangle}(w)(1+C(w))$, from which~\eqref{eq:Finv} follows.
\end{proof}

Taking now $a(\pi)=1$ for all $\pi$, as we observed:
\renewcommand{\labelitemi}{$\bullet$}
\begin{itemize}
\item $D(x)$ is the generating function for the Bell numbers, $|\Pc(n)|$.
\item $C(x)$ is the generating function for $|\CO(n)|$, and can be found from $D$ using Lemma~\ref{lem:CD}.
\item $B(x)$ is the generating function for $|\PC^+(n)|$, and can be found from $C$ using Lemma~\ref{lem:BC}.
\item $A(x)$ is the generating function for $|\PC(n)|$, and can be found from $B$ using Lemma~\ref{lem:AB}.
\end{itemize}
The Bell numbers $|\Pc(n)|$ are well known.
Using Mathematica~\cite{Wo15},
we calculated the first several terms of each generating series and we obtained the values displayed in Table~\ref{tab:ptnsizes}.
\begin{table}[ht]
\caption{Cardinalities of sets of partitions.}
\label{tab:ptnsizes}
\begin{center}
\begin{tabular}{r||r|r|r|r|} 
$n$ & $|\PC(n)|$\hspace{0.5em} & $|\PC^+(n)|$\hspace{0.14em}  & $|\CO(n)|$\hspace{0.62em}  & $|\Pc(n)|$\hspace{1.4em}
  \\ \hline\hline
$1$ & $0$ & $1$ & $1$ & $1$ \\ \hline 
$2$ & $0$ & $0$ & $1$ & $2$ \\ \hline 
$3$ & $0$ & $0$ & $1$ & $5$ \\ \hline 
$4$ & $1$ & $1$ & $2$ & $15$ \\ \hline 
$5$ & $0$ & $1$ & $6$ & $52$ \\ \hline 
$6$ & $5$ & $5$ & $21$ & $203$ \\ \hline 
$7$ & $14$ & $19$ & $85$ & $877$ \\ \hline 
$8$ & $62$ & $76$ & $385$ & $4\,140$ \\ \hline 
$9$ & $298$ & $360$ & $1\,907$ & $21\,147$ \\ \hline 
$10$ & $1\,494$ & $1\,792$ & $10\,205$ & $115\,975$ \\ \hline 
$11$ & $8\,140$ & $9\,634$ & $58\,455$ & $678\,570$ \\ \hline 
$12$ & $47\,146$ & $55\,286$ & $355\,884$ & $4\,213\,597$ \\ \hline 
$13$ & $289\,250$ & $336\,396$ & $2\,290\,536$ & $27\,644\,437$ \\ \hline 
$14$ & $1\,873\,304$ & $2\,162\,554$ & $15\,518\,391$ & $190\,899\,322$ \\ \hline 
$15$ & $12\,756\,416$ & $14\,629\,720$ & $110\,283\,179$ & $1\,382\,958\,545$ \\ \hline 
\end{tabular}
\end{center}
\end{table}

\section{Connection with random Vandermonde matrices}
\label{sec:reasons}

Purely crossing partitions arose in~\cite{BD}, appearing in the study of asymptotic moments
of certain random Vandermonde matrices $X_N$.
In particular, by Theorem 3.28 of~\cite{BD}, for the $n$-th asymptotic $*$-moment, we have
\begin{equation}\label{eq:asymptmoms}
m_n:=\lim_{N\to\infty}\mathbb{E}\circ\tr\left((X_NX_N^*)^n\right)=\sum_{\pi\in\Pc(n)}w_\pi
\end{equation}
with weight
\[
w_\pi=\tau\big(\Lambda_\pi(\underset{n-1\text{ times}}{\underbrace{1,1,\ldots,1}})\big),
\]
where $\Lambda_\pi$ is the multilinear function from $n-1$ copies of $C[0,1]$ into $C[0,1]$ described in Section~2 of~\cite{BD}
and where $\tau$ is the trace on $C[0,1]$ obtained by integrating with respect to Lebesghe measure.
These $w_\pi$ are precisely the volumes of certain polytopes first described by Ryan and Debbah in~\cite{RD09}, who had also
obtained the formula~\eqref{eq:asymptmoms}.
In Section~4 of~\cite{BD}, we show how, for arbitrary $\pi\in\Pc(n)$, $\Lambda_\pi$ and, thus $w_\pi$, can be computed
via a reduction procedure in terms of the $\Lambda_\sigma$ for $\sigma\in\bigcup_{k=1}^n\PC(k)$.
Thus procedure is akin to that used in Section~\ref{sec:partitions}, but more complicated, involving nested evaluations of various
$\Lambda_\rho$.
In this seciton, we carry out this analysis.
Let us  also mention that
noncrossing $C[0,1]$-valued cumulants for the asymptotic $*$-moments of $X_N$ are expressed in terms of purely crossing
partitions, at least for shorter lengths.
See Proposition~4.14 of~\cite{BD}.

Let $\tau$ be the trace on $C[0,1]$ given by integration with respect to Lebesgue measure.
Consider the following formal power series in variable $g\in C[0,1]$, with each $n$-th term being an $n$-fold $\Cpx$-multilinear
map of $C[0,1]\times\cdots\times C[0,1]$ into $C[0,1]$ evaluated in the variable repeated $n$ times:
\begin{alignat*}{2}
A(g)&=\sum_{n=1}^\infty a_n(g),\qquad&a_n(g)&=\sum_{\pi\in\PC(n)}\Lambda_\pi(g,\ldots,g)g \displaybreak[1]\\
B(g)&=\sum_{n=1}^\infty b_n(g),\qquad&b_n(g)&=\sum_{\pi\in\PC^+(n)}\Lambda_\pi(g,\ldots,g)g \displaybreak[2]\\
C(g)&=\sum_{n=1}^\infty c_n(g),\qquad&c_n(g)&=\sum_{\pi\in\CO(n)}\Lambda_\pi(g,\ldots,g)g \displaybreak[1]\\
D(g)&=1+\sum_{n=1}^\infty d_n(g),\qquad&d_n(g)&=\sum_{\pi\in\Pc(n)}\Lambda_\pi(g,\ldots,g)g.
\end{alignat*}
(A more general and formal treatment of such formal power series, in terms of the multilinear function series of~\cite{D07},
can be found in Section~4 of~\cite{BD:rdiag}.)
From the remarks above (see~\cite{BD} and~\cite{RD09}) if follows that
if $m_n$ is the asymptotic moment found in~\eqref{eq:asymptmoms},
then the moment generating function of $m_n$ is, for variable $x\in\Cpx$,
\[
\sum_{n=0}^\infty m_nx^n=\tau(D(x1)),
\]
where $x1\in C[0,1]$ is the constant function $x$.

The next result is analogous to the combination of Lemmas \ref{lem:AB}, \ref{lem:BC} and \ref{lem:CD}.
\begin{prop}
We have
\begin{gather}
B(g)=g+\tau(A(g))g+A(g) \label{eq:ABg} \\
C(g)=B(g/(1-\tau(g))) \label{eq:BCg} \\
D(g)=1+C(gD(g)) \label{eq:CDg}
\end{gather}
Thus, letting $F(g)=gD(g)$ and letting $F^{\langle-1\rangle}$ be its inverse with respect to composition, we have
\begin{equation}\label{eq:FDh}
F^{\langle-1\rangle}(h)=h\big(1+C(h)\big)^{-1}.
\end{equation}
\end{prop}
\begin{proof}
For the partition $0_1\in\PC(1)$, we have $\Lambda_{0_1}()=1$, so $b_1(g)=g$.
If a partition $\pi=\sigmat\in\PC^+(n)$ for $\sigma\in\PC(n-1)$ as in Lemma~\ref{lem:PC+},
then by Lemma 4.5 of~\cite{BD} we have
\[
\Lambda_\pi(\underset{n-1\text{ times}}{\underbrace{g,\ldots,g}})
=\tau\big(\Lambda_\sigma(\underset{n-2\text{ times}}{\underbrace{g,\ldots,g}})g\big).
\]
Thus, using Lemma~\ref{lem:PC+}, for $n\ge2$ we get
\[
b_n(g)=\tau(a_{n-1}(g))g+a_n(g),
\]
which yields~\eqref{eq:ABg}.

If $\pi=\Phi_n(\sigma,k_1,\ldots,k_\ell)\in\CO(n)$ for $\sigma\in\PC^+(\ell)$ and $k_j\ge1$, $k_1+\cdots+k_\ell=n$
as in Lemma~\ref{lem:NISbijection}, then by Lemma~4.4 of~\cite{BD},
we have
\[
\Lambda_\pi(\underset{n-1\text{ times}}{\underbrace{g,\ldots,g}})
=\Lambda_\sigma(\underset{\ell-1\text{ times}}{\underbrace{g,\ldots,g}})\tau(g)^{n-\ell}.
\]
Thus, using Lemma~\ref{lem:NISbijection},
\[
c_n(g)=\sum_{\ell=1}^nb_\ell(g)\sum_{\substack{k_1,\ldots,k_\ell\ge1\\k_1+\cdots+k_\ell=n}}\prod_{j=1}^\ell\tau(g)^{k_j-1}
\]
and we get
\[
C(g)=\sum_{\ell=1}^\infty b_\ell(g)\bigg(\sum_{k=1}^\infty\tau(g)^{k-1}\bigg)^\ell
=\sum_{\ell=1}^\infty b_\ell(g)/(1-\tau(g))^\ell.
\]
By multilinearity, we have $b_\ell(gx)=b_\ell(g)x^\ell$ for all scalars $x$.
Thus, we get~\eqref{eq:BCg}.

Keep in mind we use the convention $\PC(0)=\{\emptyset\}$.
Suppose $n\ge1$, $\pi\in\Pc(n)$ and $\pi=\Theta_n(\sigma,\pi_1,\ldots,\pi_\ell)$ for $\sigma\in\CO(\ell)$ and $\pi_j\in\Pc(k(j))$
where $k(j)\ge0$ and $k(1)+\cdots+k(\ell)=n-\ell$, as in Lemma~\ref{lem:PPbijection}.
Using Lemmas~4.2 and~4.3 of~\cite{BD}, we get
\[
\Lambda_\pi(\underset{n-1\text{ times}}{\underbrace{g,\ldots,g}})
=\Lambda_\sigma(e_1g,\ldots,e_{\ell-1}g)e_\ell,
\]
where
\[
e_j=\begin{cases}
g\Lambda_{\pi_j}(\underset{k(j)-1\text{ times}}{\underbrace{g,\ldots,g}}),&k(j)>0 \\
1,&k(j)=0.
\end{cases}
\]
Thus, using the convention $d_0(g)=1$, using Lemma~\ref{lem:PPbijection} and using multilinearity of $\Lambda_\sigma$, we have
\begin{align*}
D(g)&=1+\sum_{\ell=1}^\infty\;\sum_{\sigma\in\CO(\ell)}\;\sum_{k(1),\ldots,k(\ell)\ge0}
\Lambda_\sigma\big(gd_{k(1)}(g),\ldots,gd_{k(\ell-1)}(g)\big)gd_{k(\ell)}(g) \\
&=1+\sum_{\ell=1}^\infty\sum_{\sigma\in\CO(\ell)}\Lambda_\sigma(gD(g),\ldots,gD(g))gD(g)=1+C(gD(g)).
\end{align*}
This proves~\eqref{eq:CDg}.
The final equality~\eqref{eq:FDh} follows as in the proof of Lemma~\ref{lem:CD}.

\end{proof}


\begin{thebibliography}{9}

\bib{BD:rdiag}{article}{
  author={Boedihardjo, March},
  author={Dykema, Kenneth J.},
  title={On algebra-valued R-diagonal elements},
  status={preprint},
  eprint={arXiv:1512.06321},
}

\bib{BD}{article}{
  author={Boedihardjo, March},
  author={Dykema, Kenneth J.},
  title={Asymptotic $*$-moments of some random Vandermonde matrices},
  status={preprint},
  eprint={arXiv:1602.02815},
}

\bib{D07}{article}{
   author={Dykema, Kenneth J.},
   title={Multilinear function series and transforms in free probability theory},
   journal={Adv. Math.},
   volume={208},
   date={2007},
   pages={351--407},
}
		
\bib{L02}{article}{
   author={Lehner, Franz},
   title={Free cumulants and enumeration of connected partitions},
   journal={European J. Combin.},
   volume={23},
   date={2002},
   pages={1025--1031},
}

\bib{RD09}{article}{
   author={Ryan, {\O}yvind},
   author={Debbah, M{\'e}rouane},
   title={Asymptotic behavior of random Vandermonde matrices with entries on the unit circle},
   journal={IEEE Trans. Inform. Theory},
   volume={55},
   date={2009},
   pages={3115--3147},
}

\bib{Wo15}{book}{
  author={Wofram Research Inc.},
  title={Mathematica},
  edition={Version 10.3},
  publisher={Wolfram Research Inc.},
  address={Champaign, IL},
  year={2015}
}

\end{thebibliography}
\end{document}